\documentclass[reqno]{amsart}

\usepackage{mystyle}

\usepackage[headings]{fullpage}
\usepackage{parskip}

\usepackage[utf8]{inputenc}
\usepackage[T1]{fontenc}
\usepackage[british]{babel}
\usepackage[pdfencoding=auto]{hyperref}

\usepackage{microtype}
\usepackage{booktabs}

\usepackage{amsfonts,amssymb,bbm,mathrsfs,stmaryrd}
\usepackage{amsmath}

\usepackage{mathtools}
\usepackage[noabbrev]{cleveref}
\usepackage{centernot}
\usepackage[shortlabels]{enumitem}
\usepackage{url}
\usepackage{tikz}
\usepackage{pict2e}

\usetikzlibrary{matrix}
\usetikzlibrary{positioning}
\usetikzlibrary{arrows}
\tikzset{> =stealth}
\usetikzlibrary{arrows.meta}
\tikzset{normalHead/.tip={Triangle[open,angle=60:4pt]},}
\tikzset{normalTail/.tip={Triangle[reversed,open,angle=60:4pt]},}

\DeclareRobustCommand\triangleHeadArrow{
 \mathrel{\raisebox{0.5pt}{\tikz{\draw[-{Triangle[open,angle=60:5pt]}] (0,0) -- (0.4,0);}}}
}
\DeclareRobustCommand\triangleTailArrow{
 \mathrel{\tikz{\draw[>={Triangle[open,angle=60:4.5pt]}, >-{To[length=3pt]}] (0,0) -- (0.425,0);}}
}

\theoremstyle{plain}
\newtheorem{theorem}{Theorem}[section]

\newtheorem{proposition}[theorem]{Proposition}
\newtheorem{corollary}[theorem]{Corollary}

\theoremstyle{definition}
\newtheorem{definition}[theorem]{Definition}

\theoremstyle{remark}
\newtheorem{remark}[theorem]{Remark}

\renewcommand{\epsilon}{\varepsilon}
\renewcommand{\phi}{\varphi}

\renewcommand{\O}{\mathcal{O}}

\newcommand{\op}{^\mathrm{op}}

\newcommand{\id}{\mathrm{id}}

\mathchardef\mhyphen="2D

\newcommand{\Frm}{\mathrm{Frm}}
\newcommand{\RFrm}{\mathrm{Frm}_\wedge}

\newcommand{\Set}{\mathrm{Set}}
\newcommand{\Rel}{\mathrm{Rel}}

\newcommand{\Grp}{\mathrm{Grp}}
\newcommand{\SLat}{{\wedge}\mhyphen\mathrm{SLat}}
\newcommand{\Hom}{\mathrm{Hom}}
\newcommand{\Ext}{\mathrm{Ext}}
\newcommand{\SpltExt}{\mathrm{SplExt}}
\newcommand{\GenExt}{\mathrm{AdjExt}}
\newcommand{\ExtCat}{\overline{\Ext}}
\newcommand{\GenExtCat}{\overline{\GenExt}}

\DeclareMathOperator{\Aut}{Aut}
\DeclareMathOperator{\Gl}{Gl}

\newcommand{\splitext}[6]{%
\tikz[baseline]{
\newdimen{\mylabelwidth}
\newdimen{\skipwidth}
\node[anchor=base] (A) {\hspace*{\dimexpr0.5pt-\pgfkeysvalueof{/pgf/inner xsep}}${#1}$};
\settowidth{\mylabelwidth}{\pgfinterruptpicture {$#2$} \endpgfinterruptpicture}
\pgfmathsetlength{\skipwidth}{max(\mylabelwidth,10pt)}
\node[right] (B) at ([xshift=\skipwidth+12pt]A.east) {${#3}$};
\settowidth{\mylabelwidth}{\pgfinterruptpicture {$#4$} \endpgfinterruptpicture}
\settowidth{\skipwidth}{\pgfinterruptpicture {$#5$} \endpgfinterruptpicture}
\pgfmathsetlength{\skipwidth}{max(\skipwidth,\mylabelwidth,10pt)}
\node[right] (C) at ([xshift=\skipwidth+12pt]B.east) {${#6}$\hspace*{\dimexpr0.5pt-\pgfkeysvalueof{/pgf/inner xsep}}};
\draw[normalTail->] (A) to node [above] {${#2}$} (B);
\draw[transform canvas={yshift=0.5ex},-normalHead] (B) to node [above] {${#4}$} (C);
\draw[transform canvas={yshift=-0.5ex},->] (C) to node [below] {${#5}$} (B);
}}

\newcommand{\normalext}[5]{%
\tikz[baseline]{
\newdimen{\mylabelwidth}
\newdimen{\skipwidth}
\node[anchor=base] (A) {\hspace*{\dimexpr0.5pt-\pgfkeysvalueof{/pgf/inner xsep}}${#1}$};
\settowidth{\mylabelwidth}{\pgfinterruptpicture {$#2$} \endpgfinterruptpicture}
\pgfmathsetlength{\skipwidth}{max(\mylabelwidth,12pt)}
\node[right] (B) at ([xshift=\skipwidth+10pt]A.east) {${#3}$};
\settowidth{\mylabelwidth}{\pgfinterruptpicture {$#4$} \endpgfinterruptpicture}
\pgfmathsetlength{\skipwidth}{max(\mylabelwidth,10pt)}
\node[right] (C) at ([xshift=\skipwidth+10pt]B.east) {${#5}$\hspace*{\dimexpr0.5pt-\pgfkeysvalueof{/pgf/inner xsep}}};
\draw[normalTail->] (A) to node [above] {${#2}$} (B);
\draw[-normalHead] (B) to node [above] {${#4}$} (C);
}}

\title{Artin Glueings of Frames as Semidirect Products}

\author[P. F. Faul]{Peter F. Faul}
\address{Department of Pure Mathematics and Statistical Sciences\\ University of Cambridge}
\email{peter@faul.io}

\author[G. R. Manuell]{Graham R. Manuell}
\address{School of Mathematics\\ University of Edinburgh}
\email{graham@manuell.me}

\date{December 2019}

 % gluing [sic]

\begin{document}

\maketitle
\thispagestyle{empty}

\begin{abstract}
Artin glueings provide a way to reconstruct a frame from a closed sublocale and its open complement.
We show that Artin glueings can be described as the weakly Schreier split extensions in the category of frames with finite-meet preserving maps.
These extensions correspond to meet-semilattice homomorphisms between frames, yielding an extension bifunctor.
Finally, we discuss Baer sums, the induced order structure on extensions
and the failure of the split short five lemma.
\end{abstract}

\section{Introduction}

Let $H = (|H|, \O H)$ and $N = (|N|, \O N)$ be topological spaces.
We might ask which topological spaces $G$ have $H$ as an open subspace and $N$ as its closed complement. This is solved by the so-called \emph{Artin glueing} construction \cite{sga4vol1,wraith1974glueing}.
Let us briefly describe the intuition behind this construction.

In such a situation it is clear that $|G| = |N| \sqcup |H|$. Furthermore, each open $U$ in $G$ corresponds to a pair $(U_N,U_H)$ where $U_N = U \cap N \in \O N$ and $U_H = U \cap H \in \O H$. This gives an alternative description of $\O G$ as the frame $L_G$ of such pairs with the meet and join operations corresponding to componentwise intersection and union respectively.

Given such a space $G$, the pairs occurring in $L_G$ are determined by a finite-meet preserving map $\alpha \colon \O H \to \O N$ which sends each open $U \in \O H$ to the largest open $V \in \O N$ such that $(V,U) \in L_G$. A pair $(V,U)$ belongs to $L_G$ if and only if $V \subseteq \alpha(U)$. In fact, any meet-preserving map $\alpha$ will determine such a space $G$ in this way. We call the resulting space the Artin glueing of $\alpha$.

The above discussion deals only with the lattices of open sets of the topological spaces and the construction works equally well for frames. The aim of this paper is to explore the commonalities between this construction and the semidirect product of groups.

Let us recall some basic properties of semidirect products. Given a group homomorphism $\alpha \colon H \to \Aut(N)$, we can construct a group $G$ satisfying:
\begin{enumerate}
\item $H \leq G$ and $N \triangleleft G$,
\item $H \vee N = G$,
\item $H \cap N = \{e\}$.
\end{enumerate}

We see here a vague analogy between the semidirect products of groups and the Artin glueings of frames. In both cases we have objects $H$ and $N$ which we want to embed as complemented `subobjects' (sublocales in the frame case) of some other object, with $N$ normal in the group case and closed in the frame case. In both cases these constructions are entirely determined by certain structure-preserving maps involving $N$ and $H$.

In order to make this analogy precise, we look at the characterisation of semidirect products of groups as the solutions to the split extension problem. A split extension of groups is a diagram of the form \[\splitext{N}{k}{G}{e}{s}{H},\]
where $k$ is the kernel of $e$, $e$ is the cokernel of $k$, and $s$ is a section of $e$. Here $G$ will always be a semidirect product of $H$ and $N$, and the maps $k$ and $s$ will be the appropriate inclusions into the semidirect product. Furthermore, for each group $N$ there is a functor $\SpltExt(-,N) \colon \Grp\op \to \Set$ which sends a group $H$ to the set of split extensions of $H$ by $N$. This functor is naturally isomorphic to $\Hom(-,\Aut(N))$. For more details on this functor, see
\cite{borceux2005internal,borceux2005representability}.

Split extensions of groups are very well behaved and the notion of \emph{pointed protomodular category} (see \cite{bourn1998protomodularity, borceux2004protomodular}) provides a general setting in which they can be studied. Sometimes, however, only some of the split extensions in a category are well behaved and this motivates the more general idea of $\mathcal{S}$-protomodularity \cite{bourn2015Sprotomodular}, where $\mathcal{S}$ can be thought of as a collection of split extensions.

It is in this vein that we study Artin glueings. It is not possible to talk about extensions in the usual category of frames (or locales), as without zero morphisms it does not make sense to talk about kernels and cokernels. Instead we move to the category $\RFrm$ which has frames as objects and finite-meet preserving maps as morphisms. Both $\RFrm$ and $\Grp$ are full subcategories of the category of monoids. The split extensions we consider are related to the \emph{Schreier} split extensions from the study of monoids (see \cite{martins2013semidirect,bourn2015Sprotomodular}).

The relationship between $\Frm$ and $\RFrm$ is similar to the relationship between the category $\Set$ of sets and functions and the category $\Rel$ of sets and relations. In particular, $\RFrm$ is order-enriched and we can find the frame homomorphisms inside it as the left adjoints. In this way $\RFrm$ provides a proarrow equipment for $\Frm$. This category has been used alongside glueings in \cite{niefield2012glueing}.
The category $\RFrm$ can also be thought of as the category of injective meet-semilattices \cite{bruns1970injective}, though we are less sure of the implications of this.

In this paper we concern ourselves with the collection of split extensions of the form described above, but where $k$ and $s$ are required to satisfy a Schreier-type condition (or equivalently, where $s$ is required to be right adjoint to $e$). We find that $G$ will always be an Artin glueing of $H$ and $N$ determined by the map $k^*s$. As one might expect, there is a family of functors $\GenExt(-,N)$ here too, but now these extend to a bifunctor $\GenExt$, which is naturally isomorphic to $\Hom \colon \RFrm\op \times \RFrm \to \Set$. The fact that hom-sets have a natural meet-semilattice structure gives a notion of Baer sum of the extensions.
We also study the induced order structure on the extensions and the failure of the split short five lemma.
These results will be extended to the topos case in later paper.

\section{An extension problem in \texorpdfstring{$\RFrm$}{Frm∧}}

\subsection{Adjoint extensions}\label{adjoint_extensions}

The category $\RFrm$ of frames and finite-meet preserving maps is enriched over meet-semilattices and so between any two frames $L$ and $M$ there is a largest meet-preserving map. This map is the constant $1$ map, which sends each element of $L$ to the top element of $M$. It is apparent that composing with this map on either side again yields a constant $1$ map and so we see that these maps are the zero morphisms of our category. Due to this somewhat unfortunate conflict of terminology we use $\top_{L,M}$ to refer to the zero morphism between $L$ and $M$ or just $\top$ when its meaning is unambiguous.

We can now define the \emph{kernel} of a morphism $f$ as the equaliser of $f$ and $\top$ and the \emph{cokernel} of $f$ as the coequaliser of $f$ and $\top$.

\begin{definition}
A diagram of the form
\[\splitext{N}{k}{G}{e}{s}{H}\]
is called a \emph{split extension} if $k$ is the kernel of $e$, $e$ is the cokernel of $k$ and $s$ is a section of $e$. It is called an \emph{adjoint extension} if furthermore, $s$ is  right adjoint to $e$.
\end{definition}

Kernels do not always exist in this category; however, cokernels always do exist in the form described below.

\begin{proposition}\label{cokernelsexist}
Let $f \colon N \to G$ be a morphism in $\RFrm$ and let $u = f(0)$. The cokernel of $f$ is given by $e \colon G \triangleHeadArrow {\downarrow} u$, where $e(x) = x \wedge u$. Furthermore, $e$ has a right adjoint section given by $e_*(y) = y^u$.
\end{proposition}

\begin{proof}
The right adjoint to $e$ exists by well-known properties of frames. Since $e$ is surjective, $e_*$ splits $e$ by general properties of adjoints.

Clearly $e$ composes with $f$ to give $\top$ and so we need only check that it satisfies the universal property.
Suppose $g \colon G \to X$ composes with $f$ to give $\top_{N,X}$. In order to show that $e$ is the cokernel of $f$ we must show that $g$ factors through $e$ to give a unique map. Uniqueness is automatic, because $e$ is epic.

\begin{center}
   \begin{tikzpicture}[node distance=1.8cm, auto]
    \node (A) {$G$};
    \node (B) [right of=A] {${\downarrow}u$};
    \node (C) [below of=B] {$X$};
    \node(D) [left of=A] {$N$};
    \draw[transform canvas={yshift=0.5ex},->>] (A) to node {$e$} (B);
    \draw[->] (A) to node [swap]{$g$} (C);
    \draw[transform canvas={yshift=-0.5ex},->] (B) to node {$e_*$} (A);
    \draw[dashed,->] (B) to (C);
    \draw[->] (D) to node {$f$} (A);
   \end{tikzpicture}
\end{center}

To show that the meet-semilattice homomorphism $g$ factors through the surjection $e$, it is enough to show that $g(x) = g(y)$ whenever $e(x) = e(y)$. But if $e(x) = e(y)$, then $x \wedge f(0) = y \wedge f(0)$ and so $g(x) = g(x) \wedge 1 = g(x) \wedge g(f(0)) = g(x \wedge f(0)) = g(y \wedge f(0))$, which equals $g(y)$ by running the same argument in reverse.
\end{proof}

\begin{definition}
We say a morphism is a \emph{normal epimorphism} if it occurs as the cokernel of some morphism.
Dually, a monomorphism is a \emph{normal monomorphism} if it occurs as the kernel of some morphism.
\end{definition}

\Cref{cokernelsexist} shows that every normal epimorphism is of the form
$- \wedge u \colon G \triangleHeadArrow {\downarrow}u$. Conversely, such a morphism is always a normal epimorphism as it is clearly seen to be the cokernel of the inclusion of ${\uparrow}u \subseteq G$.
Note that normal epimorphisms in $\RFrm$ are precisely the open frame quotients, which accords well with the idea that $H$ should be an open sublocale of the glueing.

While $\RFrm$ does not possess all kernels, this is not a problem for working with extensions, since kernels of normal epimorphisms always do exist.
\begin{proposition}\label{kernelsexist}
Let $e \colon G \triangleHeadArrow {\downarrow}u$ be a normal epimorphism in $\RFrm$. The kernel of $e$ is given by $k \colon {\uparrow}u \triangleTailArrow G$, where $k(x) = x$. The map $k$ has a left adjoint $k^*(x) = x \vee u$ which preserves finite meets.
\end{proposition}

\begin{proof}
Let $f \colon X \to G$ compose with $e$ to give $\top_{X,H}$. The only elements sent by $e$ to 1 lie in ${\uparrow}u$ and so the image of $f$ is contained in ${\uparrow}u$.
The restriction of $f$ to ${\uparrow} u$ shows the existence condition for the universal property, while uniqueness follows since $k$ is monic.
It is then easy to see that $x \mapsto x \vee u$ provides a left adjoint to $k$.
\end{proof}
It is well known that every normal monomorphism is the kernel of its cokernel. Thus the previous result implies that a map is a normal monomorphism if and only if it is of the form ${\uparrow}u \hookrightarrow G$.
In other words, the normal monomorphisms in $\RFrm$ are precisely the right adjoints of closed frame quotients.

We now have a good understanding of both the kernel and cokernel maps in a split extension. It is only the splitting $s$ that remains mysterious. We will need to impose some conditions on the splitting in order to obtain a well-behaved theory.

Notice that for a split extension of groups, every element of $G$ is of the form $k(n)s(h)$ for some $n \in N$ and $h \in H$. For split extensions of monoids this condition does not hold automatically, but when assumed explicitly gives rise to the class of \emph{weakly Schreier} extensions \cite{bourn2015partialMaltsev,faul2019weaklySchreier}.
This is the condition we will impose --- that is, we assume that every element of the frame $G$ is of the form $k(n) \wedge s(h)$.
In fact, as we will see below, there is a \emph{canonical} choice of $n$ and $h$, since $h$ is uniquely determined and we may choose $n$ to be as large as possible.
The following result shows that under this condition, the splitting is uniquely determined by $e$.

\begin{proposition}\label{adjointsplitting}
A split extension \splitext{N}{k}{G}{e}{s}{H} in $\RFrm$
is weakly Schreier if and only if it is an adjoint extension.
\end{proposition}

\begin{proof}
By \cref{cokernelsexist,kernelsexist} we may take $N$ to be of the form ${\uparrow} u$ and $H$ to be ${\downarrow} u$ for some element $u \in G$. Then $k(x) = x$ and $e(x) = x \wedge u$. 

Suppose the extension is weakly Schreier and consider an element $g \in G$. By assumption, $g = k(n) \wedge s(h)$ for some $n \in N$ and some $h \in H$. Then $e(g) = ek(n) \wedge es(h) = 1 \wedge h = h$. Thus, $g = k(n) \wedge se(g)$. In particular, $g \le se(g)$ and so $\id_G \le se$. But $es = \id_H \le \id_H$ and so $s$ is right adjoint to $e$ as required.

For the other direction suppose $s$ is the right adjoint of $e$ so that $s(x) = e_*(x) = x^u$. We must show that each element of $g \in G$ can be expressed as $k(n) \wedge e_*(h)$ for some $n \in N$ and $h \in H$. By the above we may take $h = e(g)$, while $k^*(g)$ is the most natural candidate for $n$. Taking the meet yields $kk^*(g) \wedge e_*e(g) = (g \vee u) \wedge (g \wedge u)^u = (g \vee u) \wedge g^u \wedge u^u = (g \vee u) \wedge g^u = (g \wedge g^u) \vee (u \wedge g^u) = g \vee (u \wedge g^u)$. But $u \wedge g^u \le g$ and thus $kk^*(g) \wedge e_*e(g) = g$ as required.
\end{proof}

This means that all the information of a weakly Schreier extension \splitext{N}{k}{G}{e}{s}{H} is contained in the normal epi $e$, since $k$ can be recovered as its kernel and $s$ as its right adjoint. Since every normal epi gives rise to a weakly Schreier extension, we have a complete classification of the weakly Schreier extensions in $\RFrm$.
From now on, we will refer to this class of split extensions as the \emph{adjoint extensions}.

Let $\mathcal{S}$ be the class of normal epimorphisms $\RFrm$ equipped with their left adjoints. It is clear that all isomorphisms belong to $\mathcal{S}$ and we will show $\mathcal{S}$ is stable under pullback in \cref{ExtFunctor1}. Finally, for any such map, the adjoint and the kernel are jointly extremally epic by \cref{adjointsplitting}. Thus, the pointed category $\RFrm$ is $\mathcal{S}$-protomodular in the sense of \cite{bourn2015partialMaltsev}. That is, it is almost an $\mathcal{S}$-protomodular category in the sense of \cite{bourn2015Sprotomodular} except for the requirement of finite-completeness.

Notice also that any \emph{extension} in $\RFrm$ --- that is, a diagram \normalext{N}{k}{G}{e}{H}
in which $e$ is the cokernel of $k$, and $k$ is the kernel of $e$ --- gives rise to a unique adjoint extension and vice versa. Thus the adjoint extension problem coincides with the extension problem in $\RFrm$. However, as we will see later these concepts should not conflated, since the morphisms of extensions and (adjoint) split extensions are not the same.

\subsection{Artin glueings}

In the introduction we gave an informal description of the Artin glueing construction. We now discuss them in more detail and explore the connection to adjoint extensions. The ideas behind the results of this section are well understood (for example, see \cite{niefield1981cartesian}), but we prove them in our current context for completeness.

\begin{definition}
The \emph{Artin glueing} of two frames $H$ and $N$ along a finite-meet preserving map $\alpha \colon H \to N$ is given by the frame \[\Gl(\alpha) = \{(n,h) \in N \times H \mid n \le \alpha(h)\}\]
equipped with projections $\pi_1 \colon \Gl(\alpha) \to N$ and $\pi_2 \colon \Gl(\alpha) \to H$. Here the finite meet and arbitrary join operations in $\Gl(\alpha)$ are taken componentwise.
\end{definition}
Since meets in $\Gl(\alpha)$ are computed componentwise, we see that $\pi_1$ and $\pi_2$ preserve finite meets and so are morphisms in $\RFrm$. In fact, they both have right adjoints in $\RFrm$. The right adjoint of $\pi_1$ is given by ${\pi_1}_*(n) = (n,1)$ and the right adjoint of $\pi_2$ is given by ${\pi_2}_*(h) = (\alpha(h),h)$.

With these right adjoints, Artin glueings give rise to adjoint extensions.
\begin{proposition}\label{glueings_are_adjoint_extensions}
Let $\alpha \colon H \to N$ be a finite-meet preserving map. The diagram \[\splitext{N}{{\pi_1}_*}{\Gl(\alpha)}{\pi_2}{{\pi_2}_*}{H}\]
is an adjoint extension.
\end{proposition}

\begin{proof}
It is enough to show that $\pi_2$ is a normal epi and that ${\pi_1}_*$ is its kernel. Observe that $H$ is isomorphic to ${\downarrow}(0,1) \subseteq \Gl(\alpha)$ via the map $h \mapsto (0,h)$ and that this makes the following diagram commute.
\begin{center}
   \begin{tikzpicture}[node distance=1.8cm, auto]
    \node (A) {$\Gl(\alpha)$};
    \node (B) [right of=A] {$H$};
    \node (C) [below of=B] {${\downarrow}(0,1)$};
    \draw[->] (A) to node {$\pi_2$} (B);
    \draw[-normalHead] (A) to node [swap]{${-} \wedge (0,1)$} (C);
    \draw[->] (B) to node [below,rotate=90,xshift=1pt] {$\sim$} (C);
   \end{tikzpicture}
\end{center}
The map ${-} \wedge (0,1)$ is a normal epi and hence so is $\pi_2$.
Furthermore, the kernel of ${-} \wedge (0,1)$ is ${\uparrow}(0,1) \hookrightarrow \Gl(\alpha)$, which is clearly isomorphic to ${\pi_1}_* \colon N \to \Gl(\alpha)$.
\end{proof}

Notice that $\alpha$ can be recovered from the glueing by considering the composite  $\pi_1{\pi_2}_* = \alpha$.
In fact, any extension \splitext{N}{k}{G}{e}{e_*}{H} gives rise a finite-meet preserving map $k^*e_*$ and we may glue along this map to obtain another extension as above.

In order to compare the original adjoint extension to the one given by the glueing of $k^*e_*$ we require a notion of morphism of adjoint extensions.
\begin{definition}\label{morphismdefinition}
A \emph{morphism of adjoint extensions} from $\splitext{N}{k}{G}{e}{e_*}{H}$
to $\splitext{N}{k'}{G'}{e'}{e'_*}{H}$ is a meet-preserving map $f \colon G \to G'$ making the three squares in the following diagram commute.
\begin{center}
   \begin{tikzpicture}[node distance=2cm, auto]
    \node (A) {$N$};
    \node (B) [right of=A] {$G$};
    \node (C) [right of=B] {$H$};
    \node (D) [below of=A] {$N$};
    \node (E) [right of=D] {$G'$};
    \node (F) [right of=E] {$H$};
    \draw[normalTail->] (A) to node {$k$} (B);
    \draw[transform canvas={yshift=0.5ex},-normalHead] (B) to node {$e$} (C);
    \draw[transform canvas={yshift=-0.5ex},->] (C) to node {$e_*$} (B);
    \draw[normalTail->] (D) to node {$k'$} (E);
    \draw[transform canvas={yshift=0.5ex},-normalHead] (E) to node {$e'$} (F);
    \draw[transform canvas={yshift=-0.5ex},->] (F) to node {$e_*'$} (E);
    \draw[->] (B) to node {$f$} (E);
    \draw[double equal sign distance] (A) to (D);
    \draw[double equal sign distance] (C) to (F);
   \end{tikzpicture}
\end{center}
Explicitly, we require $fk = k'$, $e'f = e$ and $fe_* = e'_*$.
\end{definition}
It is apparent that isomorphisms of adjoint extensions are those for which the meet-preserving map is an isomorphism.

\begin{theorem}\label{MainResult}
The adjoint extensions \splitext{N}{k}{G}{e}{e_*}{H} and \splitext{N}{{\pi_1}_*}{\Gl(k^*e_*)}{\pi_2}{{\pi_2}_*}{H}
are isomorphic.
\end{theorem}

\begin{proof}
 Without loss of generality we may assume $N = {\uparrow} u$, $H = {\downarrow} u$, $k(x) = x$, $e(x) = x \wedge u$ and $e_*(x) = x^u$.
 
 Consider the maps $f \colon G \to \Gl(k^*e_*)$ and $f' \colon \Gl(k^*e_*) \to G$ given by $f(g) = (k^*(g),e(g))$ and $f'(n,h) = k(n) \wedge e_*(h)$.
 Notice that the map $f$ is well defined because $g \leq e_*e(g)$ and so $k^*(g) \leq k^*e_*e(g)$. Clearly both $f$ and $f'$ preserve finite meets.
 
 \begin{center}
   \begin{tikzpicture}[node distance=2.2cm, auto]
    \node (A) {$N$};
    \node (B) [right of=A] {$G$};
    \node (C) [right of=B] {$H$};
    \node (D) [below of=A] {$N$};
    \node (E) [right of=D] {$\Gl(k^*e_*)$};
    \node (F) [right of=E] {$H$};
    \draw[normalTail->] (A) to node {$k$} (B);
    \draw[transform canvas={yshift=0.5ex},-normalHead] (B) to node {$e$} (C);
    \draw[transform canvas={yshift=-0.5ex},->] (C) to node {$e_*$} (B);
    \draw[normalTail->] (D) to node {${\pi_1}_*$} (E);
    \draw[transform canvas={yshift=0.5ex},-normalHead] (E) to node {$\pi_2$} (F);
    \draw[transform canvas={yshift=-0.5ex},->] (F) to node {${\pi_2}_*$} (E);
    \draw[transform canvas={xshift=0.5ex},->] (B) to node {$f$} (E);
    \draw[transform canvas={xshift=-0.5ex},->] (E) to node {$f'$} (B);
    \draw[double equal sign distance] (A) to (D);
    \draw[double equal sign distance] (C) to (F);
   \end{tikzpicture}
\end{center}
 
 We claim these maps are inverses. First note $f'f(g) = kk^*(g) \wedge e_*e(g) = g$, where the final equality follows as in the proof of \cref{adjointsplitting}. Next we have $ff'(n,h) = (k^*(k(n) \wedge e_*(h)), e(k(n) \wedge e_*(h))) = (k^*k(n) \wedge k^*e_*(h), ek(n) \wedge ee_*(h))$. Notice that $ek(n) = 1$ and $ee_*(h) = h$ and so the second component is $h$ as required. Next observe that $k^*k(n) = n$ and $n \leq k^*e_*(h)$, since $(n,h) \in \Gl(k^*e_*)$. This gives that the meet in the first component is $n$ as required. Thus $G$ and $\Gl(k^*e_*)$ are isomorphic in $\RFrm$.
 
 It remains to show that $f$ makes the appropriate squares commute. Firstly, $fk(n) = (k^*k(n),ek(n)) = (n,1) = {\pi_1}_*(n)$ as required. Next notice $\pi_2 f(g) = \pi_2(k^*(g),e(g)) = e(g)$ as required. Finally, $fe_*(h) = (k^*e_*(h),ee_*(h)) = (k^*e_*(h),h) = {\pi_2}_*(h)$ as required.
\end{proof}

So, similarly to the case of groups where the split extensions could be identified with maps $\alpha \colon H \to \Aut(N)$, we see that the adjoint extensions between $H$ and $N$ correspond to finite-meet preserving maps $\beta \colon H \to N$.

Using \cref{adjointsplitting}, we can also state \cref{MainResult} in terms of weakly Schreier extensions.

\begin{corollary}
The Artin glueings of $H$ and $N$ are in one-to-one correspondence with the weakly Schreier extensions of $H$ by $N$ in $\Frm_\wedge$.
\end{corollary}

This description of weakly Schreier extensions is generalised to the case of monoids in \cite{faul2019weaklySchreier}.

\section{Extension functors}

\subsection{Functoriality of adjoint extensions}

In the category of groups (or any protomodular category with semidirect products --- see \cite{bourn1998protomodularity} and \cite{borceux2005internal}) there is a functor $\SpltExt(-,N)$ for each object $N$ which sends an object $H$ to the set of isomorphism classes of split extensions of $H$ by $N$. This functor acts on morphisms by pullback as described below.

Let $f \colon H' \to H$ be morphism and suppose \splitext{N}{k}{G}{e}{s}{H} is a split extension of groups. Consider the pullback $e'$ of $e$ along $f$.
\begin{center}
  \begin{tikzpicture}[node distance=2.5cm, auto]
    \node (A) {$G \times_{\scriptscriptstyle H} H'$};
    \node (B) [below of=A] {$G$};
    \node (C) [right of=A] {$H'$};
    \node (D) [right of=B] {$H$};
    \node (E) [left of=B] {$N$};
    \node (F) [left of=A] {$N$};
    \draw[->] (A) to node [swap] {$f'$} (B);
    \draw[transform canvas={yshift=0.5ex},-normalHead] (A) to node {$e'$} (C);
    \draw[transform canvas={yshift=-0.5ex},->] (C) to node {$s'$} (A);
    \draw[transform canvas={yshift=0.5ex},-normalHead] (B) to node {$e$} (D);
    \draw[transform canvas={yshift=-0.5ex},->] (D) to node {$s$} (B);
    \draw[->] (C) to node {$f$} (D);
    \draw[normalTail->] (F) to node {$k'$} (A);
    \draw[normalTail->] (E) to node [swap] {$k$} (B);
    \draw[double equal sign distance] (E) to (F);
    \begin{scope}[shift=({A})]
        \draw +(0.3,-0.7) -- +(0.7,-0.7) -- +(0.7,-0.3);
    \end{scope}
  \end{tikzpicture}
\end{center}
The kernel $k'$ of $e'$ has domain $N$ and together with $e'$ forms an extension. The universal property of the pullback then gives a map $s'$ as a canonical choice of splitting of $e'$. We define $\SpltExt(f,N) \colon \SpltExt(H,N) \to \SpltExt(H',N)$ to be the function sending the original extension to the one described above. 

We now show that the adjoint extensions in $\RFrm$ have associated functors $\GenExt(-,N)$ in a similar way.

\begin{proposition}\label{ExtFunctor1}
Let $e \colon G \triangleHeadArrow H$ be a normal epimorphism and $k$ its kernel. The pullback of $e$ along a morphism $f \colon H' \to H$ exists and is a normal epimorphism of the form $\pi'_2 \colon \Gl(k^*e_*f) \triangleHeadArrow H'$.
\end{proposition}

\begin{proof}
It is sufficient to compute the pullback in the category $\SLat$ of meet-semilattices and to observe that it is a frame. By \cref{MainResult} we can take $e$ to be $\pi_2 \colon \Gl(k^*e_*) \to H$. Since $\SLat$ is algebraic, the pullback of $\pi_2$ and $f$ is given by $\Gl(k^*e_*) \times_{\scriptscriptstyle H} H' = \{((n,h),h') \in \Gl(k^*e_*) \times H' : \pi_2(n,h) = f(h')\}$
with projections $p_1$ and $p_2$.
\begin{center}
  \begin{tikzpicture}[node distance=2.5cm, auto]
    \node (A) {$\Gl(k^*e_*) \times_{\scriptscriptstyle H} H'$};
    \node (B) [below of=A] {$\Gl(k^*e_*)$};
    \node (C) [right of=A] {$H'$};
    \node (D) [right of=B] {$H$};
    \draw[->] (A) to node [swap] {$p_1$} (B);
    \draw[->] (A) to node {$p_2$} (C);
    \draw[-normalHead] (B) to node [swap] {$\pi_2$} (D);
    \draw[->] (C) to node {$f$} (D);
    \begin{scope}[shift=({A})]
        \draw +(0.3,-0.7) -- +(0.7,-0.7) -- +(0.7,-0.3);
    \end{scope}
  \end{tikzpicture}
 \end{center}

Because $\pi_2(n,h) = h$, the element $((n,h),h')$ belongs to the pullback if and only if $f(h') = h$. Combining this with the fact that $(n,h) \in \Gl(k^*e_*)$ we have that $n \leq k^*e_*f(h')$. It is then easy to see that $((n,h),h')$ belongs to the pullback if and only if $(n,h') \in \Gl(k^*e_* f)$. This induces an obvious isomorphism giving that $\pi'_2 \colon \Gl(k^*e_* f) \to H'$ is the pullback of $e$ along $f$. This map is a normal epi by \cref{glueings_are_adjoint_extensions}.
\end{proof}

This result allows us to define a family of functors $(\GenExt(-,N))_{N \in \RFrm}$, where $\GenExt(H,N)$ is the set of isomorphism classes of adjoint extensions of $H$ by $N$ and $\GenExt(f,N)\colon \GenExt(H,N) \to \GenExt(H', N)$ sends the adjoint extension \splitext{N}{k}{G}{e}{e_*}{H} to \splitext{N}{{\pi'_1}_*}{\Gl(k^*e_*f)}{\pi'_2}{{\pi'_2}_*}{H'}.

Recall that in the case of groups $\SpltExt(-,N)$ is representable with $\Aut(N)$ as the representing object. \Cref{ExtFunctor1,MainResult} together give a similar result in our setting.

\begin{corollary}\label{natiso1}
For each frame $N$, the functor $\GenExt(-,N) \colon \RFrm\op \to \Set$ is naturally isomorphic to $\Hom(-,N) \colon \RFrm\op \to \Set$. 
\end{corollary}

In fact, things work even better in our situation than in the group case.
It is natural to ask whether the family of functors $(\SpltExt(-,N))_{N \in \Grp}$ assemble into a bifunctor $\SpltExt \colon \Grp\op \times \Grp \to \Set$.
In particular, we could ask if there is a functor $\SpltExt(H,-)$ for each group $H$. The functor $\SpltExt(-,N)$ is computed by taking a pullback, so we might expect $\SpltExt(H,-)$ is given by a pushout. However, this does not work as smoothly as in the former case. In fact, $(\SpltExt(-,N))_{N \in \Grp}$ cannot be extended to a bifunctor at all, as can be seen from the isomorphism $\SpltExt(-,N) \cong \Hom(-,\Aut(N))$, the Yoneda lemma and the fact that $\Aut(-)$ cannot be extended to a functor.

However, in the frame case the isomorphism $\GenExt(-,N) \cong \Hom(-,N)$ shows that that the family $(\GenExt(-,N))_{N \in \RFrm}$ does extend to a bifunctor in an obvious way. Below we show that unlike in the group case the pushout construction for $\GenExt(H,-)$ succeeds.

\begin{proposition}\label{ExtFunctor2}
Let $k \colon N \triangleTailArrow G$ be a normal monomorphism and $e \colon G \triangleHeadArrow H$ its cokernel. Let $f \colon N \to N'$ be a morphism. The pushout of $k$ along $f$ exists and is a normal monomorphism of the form ${\pi'_1}_* \colon N' \triangleTailArrow \Gl(fk^*e_*)$.
\end{proposition}

\begin{proof}
By \cref{MainResult} we may take $k$ to be the map ${\pi_1}_* \colon N \to \Gl(k^*e_*)$. We claim that the pushout of ${\pi_1}_*$ and $f$ is given by $\Gl(fk^*e_*)$ with the injections ${\pi'_1}_* \colon N' \to \Gl(fk^*e_*)$ and $f' \colon \Gl(k^*e_*) \to \Gl(fk^*e_*)$, where $f'(n,h) = (f(n),h)$, as in the diagram below.
\begin{center}
  \begin{tikzpicture}[node distance=2.5cm, auto]
    \node (A) {$N$};
    \node (B) [below of=A] {$N'$};
    \node (C) [right of=A] {$\Gl(k^*e_*)$};
    \node (D) [right of=B] {$\Gl(fk^*e_*)$};
    \draw[->] (A) to node [swap] {$f$} (B);
    \draw[normalTail->] (A) to node {${\pi_1}_*$} (C);
    \draw[->] (B) to node {${\pi'_1}_*$} (D);
    \draw[->] (C) to node [swap] {$f'$} (D);
    \begin{scope}[shift=({D})]
        \draw +(-0.3,0.7) -- +(-0.7,0.7) -- +(-0.7,0.3);
    \end{scope}
    \node (X) [below right=1.2cm and 1.2cm of D.center] {$X$};
    \draw[bend right=15,->] (B) to node [swap] {$q$} (X);
    \draw[bend left=15,->] (C) to node {$p$} (X);
    \draw[dashed,->] (D) to node [swap,pos=0.4] {$\ell$} (X);
  \end{tikzpicture}
 \end{center}

Firstly note that $f'$ is well defined, since if $n \leq k^*e_*(h)$, then $f(n) \leq fk^*e_*(h)$. Next notice that 
$f'{\pi_1}_*(n) = f'(n,1) = (f(n),1) = {\pi'_1}_*f(n)$
and so the diagram commutes.

We show the uniqueness condition of the universal property first. Suppose $p \colon \Gl(k^*e_*) \to X$ and $q \colon N' \to X$ together form a cocone and that there is a morphism $\ell \colon \Gl(fk^*e_*) \to X$ such that $\ell f' = p$ and $\ell{\pi'_1}_* = q$. This means that $\ell(f(n),h) = p(n,h)$ and $\ell(n,1) = q(n)$. Notice that each element $(n,h) \in \Gl(fk^*e_*)$ can be written as $(n,h) = (fk^*e_*(h),h) \wedge (n,1)$ and so $\ell$ is uniquely determined by
$\ell(n,h) = \ell(fk^*e_*(h),h) \wedge \ell(n,1) = p(k^*e_*(h),h) \wedge q(n)$.

Note that the map $\ell$ as defined above does indeed preserve finite meets.
So we need only show that $p$ and $q$ factor through $\ell$ as required. For $q$ we have the simple equality $\ell{\pi'_1}_*(n) = \ell(n,1) = p(1,1) \wedge q(n) = q(n)$.

For $p$ we must show that $\ell(f(n),h) = p(n,h)$. We know that $\ell(f(n),h) = p(fk^*e_*(h),h) \wedge q(f(n))$. Because $p$ and $q$ form a cocone, we have that $q(f(n)) = p{\pi_1}_*(n) = p(n,1)$. Substituting this in we get $\ell(f(n),h) = p(fk^*e_*(h),h) \wedge p(n,1) = p(fk^*e_*(h),h) \wedge (n,1)) = p(n,h)$
as required.
\end{proof}

This allows us to define a functor $\GenExt(H,-)$ for each frame $H$. For a map $f \colon N \to N'$, the resulting map $\GenExt(H,f)$ sends the adjoint extension \splitext{N}{k}{G}{e}{e_*}{H} to \splitext{N}{{\pi'_1}_*}{\Gl(k^*e_*f)}{\pi'_2}{{\pi'_2}_*}{H'}. We then have the following corollary as above.

\begin{corollary}\label{natiso2}
The functor $\GenExt(H,-) \colon \RFrm \to \Set$ is naturally isomorphic to $\Hom(H,-) \colon \RFrm\op \to \Set$.
\end{corollary}

To show the families $(\GenExt(-,N))_{N \in \RFrm}$ and $(\GenExt(H,-))_{H \in \RFrm}$ yield a bifunctor, we only need that $\GenExt(H',g)\GenExt(f,N) = \GenExt(f,N')\GenExt(H,g)$ and to set $\GenExt(f,g)$ to be their common value.
By \cref{natiso1,natiso2} we know that each family is isomorphic to hom functors, which clearly satisfy the condition and hence $\GenExt$ is a bifunctor naturally isomorphic to $\Hom$. We record this in the following theorem.

\begin{theorem}\label{ExtisoHom}
There is a bifunctor $\GenExt \colon \RFrm\op \times \RFrm \to \Set$ where $\GenExt(H,N)$ the set of isomorphism classes of adjoint extensions of $H$ by $N$, $\GenExt(f,N)$ is given by pullback along $f$ and $\GenExt(H,g)$ is given by pushout along $g$ as above.
This is naturally isomorphic to $\Hom \colon \RFrm\op \times \RFrm \to \Set$.
\end{theorem}

\subsection{Additional structure on the set of adjoint extensions}

Since $\RFrm$ is enriched over meet-semilattices, the natural isomorphism of \cref{ExtisoHom} induces a meet-semilattice structure on $\GenExt(H,N)$. On the other hand, the extensions of $H$ by $N$ have a natural category structure using the morphisms defined in \cref{morphismdefinition}.
It is natural to ask how these two structures might be related. Let us start by considering the first notion in more detail.

Recall that an extension is a diagram $\normalext{N}{k}{G}{e}{H}$ in which $k$ is the kernel of $e$ and $e$ is the cokernel of $k$. As discussed at the end of \cref{adjoint_extensions}, every extension in $\RFrm$ admits a unique splitting that turns it into an adjoint extension. Thus far we have mainly focused on describing the elements of $\GenExt(H, N)$ as isomorphism classes of adjoint extensions in order to make an analogy with split extensions of groups.
But thinking of them in terms of \emph{extensions} instead allows us to make a different analogy --- this time to extensions of \emph{abelian} groups. This shift of perspective makes sense because isomorphisms of extensions and adjoint extensions agree.

In an abelian category, the $\Ext$ functor admits a natural abelian group structure. Here the binary operation is called the \emph{Baer sum} of extensions. More generally, in any category with biproducts every object has a unique commutative monoid structure and so any product-preserving $\Set$-valued functor on such a category factors through the category of commutative monoids. When applied to the $\Ext$ functor of an abelian category, this yields the Baer sum operation. In our situation, $\GenExt$ preserves finite products (as it is isomorphic to $\Hom$) and this construction endows $\GenExt(H, N)$ with the structure of a meet-semilattice.

Of course, as mentioned above we can obtain the same result more directly by applying the isomorphism $\GenExt \cong \Hom$ and using the natural meet-semilattice structure on the hom-sets. Explicitly, this gives that $\normalext{N}{{\pi_1}_*}{N \times H}{\pi_2}{H}$ is the top element and the meet of extensions $\normalext{N}{k_1}{G_1}{e_1}{H}$ and $\normalext{N}{k_2}{G_2}{e_2}{H}$ is given by $\normalext{N}{{\pi_1}_*}{\Gl(k_1^*{e_1}_* \wedge k_2^*{e_2}_*)}{\pi_2}{H}$. We can then apply the following proposition to give a particularly concrete description of the meet.

\begin{proposition} \label{meetOfExtensions}
Let $\alpha, \beta \colon H \to N$ be meet-semilattice homomorphisms. Then $\Gl(\alpha \wedge \beta)$ is given by the intersection of $\Gl(\alpha)$ and $\Gl(\beta)$ as sub-meet-semilattices of $N \times H$ in the obvious way. 
\end{proposition}
\begin{proof}
Simply observe that
\begin{align*}
\Gl(\alpha \wedge \beta) &= \{(n,h) \in N \times H \mid n \leq \alpha(h) \wedge \beta(h)\} \\
                         &= \{(n,h) \in N \times H \mid n \leq \alpha(h)\} \cap \{(n,h) \in N \times H \mid n \leq \beta(h)\} \\
                         &= \Gl(\alpha) \cap \Gl(\beta),
\end{align*}
where the final intersection is taken in $N \times H$.
\end{proof}

It is also interesting to consider the ordering of extensions induced by this meet-semilattice structure. 
\begin{corollary}\label{glueingorder}
If $\alpha \le \beta$, then there is an obvious inclusion $\Gl(\alpha) \hookrightarrow \Gl(\beta)$.
\end{corollary}

So whenever an extension $\normalext{N}{k_1}{G_1}{e_1}{H}$ is less than or equal to $\normalext{N}{k_2}{G_2}{e_2}{H}$ in $\Ext(N,H)$, we can view $G_1$ as a subset of $G_2$. In fact, we obtain a morphism of extensions in the following sense.

\begin{definition}
The category $\ExtCat(H,N)$ has extensions of $H$ by $N$ as objects and commutative diagrams of the form
\begin{center}
   \begin{tikzpicture}[node distance=2.0cm, auto]
    \node (A) {$N$};
    \node (B) [right of=A] {$G$};
    \node (C) [right of=B] {$H$};
    \node (D) [below of=A] {$N$};
    \node (E) [right of=D] {$G'$};
    \node (F) [right of=E] {$H$};
    \draw[normalTail->] (A) to node {$k$} (B);
    \draw[-normalHead] (B) to node {$e$} (C);
    \draw[normalTail->] (D) to node {$k'$} (E);
    \draw[-normalHead] (E) to node {$e'$} (F);
    \draw[->] (B) to node {$f$} (E);
    \draw[double equal sign distance] (A) to (D);
    \draw[double equal sign distance] (C) to (F);
   \end{tikzpicture}
\end{center}
as morphisms.
\end{definition}
The meet-semilattice structure on $\Hom(G, G')$ induces a meet-semilattice structure on the hom-sets of the category $\ExtCat(H, N)$ and so, in particular, the category is order-enriched.

Compare this to the category of \emph{adjoint} extensions alluded to at the start of this section.

\begin{definition}
The category $\GenExtCat(H,N)$ has adjoint extensions as objects and
morphisms given by diagrams of the form below as in \cref{morphismdefinition}.
\begin{center}
   \begin{tikzpicture}[node distance=2.0cm, auto]
    \node (A) {$N$};
    \node (B) [right of=A] {$G$};
    \node (C) [right of=B] {$H$};
    \node (D) [below of=A] {$N$};
    \node (E) [right of=D] {$G'$};
    \node (F) [right of=E] {$H$};
    \draw[normalTail->] (A) to node {$k$} (B);
    \draw[transform canvas={yshift=0.5ex},-normalHead] (B) to node {$e$} (C);
    \draw[transform canvas={yshift=-0.5ex},->] (C) to node {$e_*$} (B);
    \draw[normalTail->] (D) to node {$k'$} (E);
    \draw[transform canvas={yshift=0.5ex},-normalHead] (E) to node {$e'$} (F);
    \draw[transform canvas={yshift=-0.5ex},->] (F) to node {$e_*'$} (E);
    \draw[->] (B) to node {$f$} (E);
    \draw[double equal sign distance] (A) to (D);
    \draw[double equal sign distance] (C) to (F);
   \end{tikzpicture}
\end{center}
\end{definition}
Note that, unlike morphisms of extensions, morphisms of adjoint extensions are required to preserve the splittings. So while there is a correspondence between the objects, they give different categories. In particular, the map given in \cref{glueingorder} need not be a morphism of adjoint extensions, but in the following theorem we will see that it has a right adjoint which is.

\begin{theorem} \label{category_adjext_poset}
There is an equivalence of categories between the category $\GenExtCat(H,N)$ and the underlying poset of $\GenExt(H, N)$ with the \emph{reverse} order.
\end{theorem}

\begin{proof}
Let $\alpha, \beta\colon H \to N$ be meet-semilattice homomorphisms such that $\alpha \le \beta$.
Then there is a corresponding inclusion $i\colon \Gl(\alpha) \hookrightarrow \Gl(\beta)$ as in \cref{glueingorder}. This inclusion is a frame homomorphism and thus has a right adjoint $i_*\colon \Gl(\beta) \to \Gl(\alpha)$. This adjoint sends $(n,h)$ to $(n \wedge \alpha(h),h)$ and can be seen to be a morphism of adjoint extensions. 

This procedure defines the action on morphisms of a contravariant functor from the underlying poset of $\GenExt(H,N)$ to $\GenExtCat(H,N)$. This functor is essentially surjective by \cref{MainResult} and it is automatically faithful, since $\GenExt(H, N)$ is a poset. We now show it is full.

Suppose $f$ is a morphism of adjoint extensions as given below.

\begin{center}
   \begin{tikzpicture}[node distance=2cm, auto]
    \node (A) {$N$};
    \node (B) [right of=A] {$\Gl(\beta)$};
    \node (C) [right of=B] {$H$};
    \node (D) [below of=A] {$N$};
    \node (E) [right of=D] {$\Gl(\alpha)$};
    \node (F) [right of=E] {$H$};
    \draw[normalTail->] (A) to node {$k$} (B);
    \draw[transform canvas={yshift=0.5ex},-normalHead] (B) to node {$e$} (C);
    \draw[transform canvas={yshift=-0.5ex},->] (C) to node {$e_*$} (B);
    \draw[normalTail->] (D) to node {$k'$} (E);
    \draw[transform canvas={yshift=0.5ex},-normalHead] (E) to node {$e'$} (F);
    \draw[transform canvas={yshift=-0.5ex},->] (F) to node {$e_*'$} (E);
    \draw[->] (B) to node {$f$} (E);
    \draw[double equal sign distance] (A) to (D);
    \draw[double equal sign distance] (C) to (F);
   \end{tikzpicture}
\end{center}

We know that $fe_* = e'_*$, which means $f(\beta(h),h) = (\alpha(h),h)$, and that $fk = k'$, which means $f(n,1) = (n,1)$. Since any element $(n,h) \in \Gl(\beta)$ can be written $(n,h) = (n,1) \wedge (\beta(h),h)$, we see that $f(n,h) = (n \wedge \alpha(h),h)$.

Since $f$ is in the form described above, we may conclude that it is in the image of the functor once we show that $\alpha \le \beta$. 
To see this note that $(\alpha(h),h) = f(\beta(h),h) = (\beta(h) \wedge \alpha(h), h)$.
\end{proof}

In the above proof we used the characterisation of equivalences of categories in terms of essentially surjective fully faithful functors for simplicity.
However, it also possible to construct the `inverse' equivalence explicitly without appealing to the axiom of choice.

\begin{remark}
\Cref{category_adjext_poset} shows that in contrast to the case of groups, the \emph{split short five lemma} fails in the $S$-protomodular category $\RFrm$. However, in \cite{bourn2015Sprotomodular} it is shown that the split short five lemma holds in any $S$-protomodular category. Of course, in that paper there is an additional assumption on the class $S$ which is not satisfied here.

Nonetheless, the category $\GenExtCat(H,N)$ still has a special form, since it is a preorder. In fact, this will always be the case for this more general notion of $S$-protomodular category (using the joint epicness of the kernel and the splitting).
\end{remark}

With a little more work, it is also possible to relate the order on $\GenExt(H,N)$ to the category of extensions $\ExtCat(H,N)$. 
Recall that $\ExtCat(H, N)$ is order-enriched and so we may consider adjoint morphisms in the category.

\begin{proposition}
 There is an equivalence of categories between $\GenExtCat(H,N)$
 and the category of adjunctions of $\ExtCat(H, N)$, where we take the morphisms to be in the direction of the \emph{right} adjoint.
\end{proposition}

\begin{proof}
We claim that there is a fully faithful bijective-on-objects functor from the category of adjunctions on $\ExtCat(H,N)$ to $\GenExtCat(H,N)$, which sends each
extension to the corresponding adjoint extension and each pair of adjoints $(f,f^*)$ to the right adjoint $f$.

Suppose $f \colon G \to G'$ is a morphism of extensions and that $f$ has a left adjoint which is also a morphism of extensions, as shown in the following diagram.

\begin{center}
   \begin{tikzpicture}[node distance=2.0cm, auto]
    \node (A) {$N$};
    \node (B) [right of=A] {$G$};
    \node (C) [right of=B] {$H$};
    \node (D) [below of=A] {$N$};
    \node (E) [right of=D] {$G'$};
    \node (F) [right of=E] {$H$};
    \draw[normalTail->] (A) to node {${\pi_1}_*$} (B);
    \draw[-normalHead] (B) to node {$\pi_2$} (C);
    \draw[normalTail->] (D) to [swap] node {${\pi'_1}_*$} (E);
    \draw[-normalHead] (E) to [swap] node {$\pi'_2$} (F);
    \draw[transform canvas={xshift=-0.5ex},->] (B) to [swap] node {$f$} (E);
    \draw[transform canvas={xshift=0.5ex},->] (E) to [swap] node {$f^*$} (B);
    \draw[double equal sign distance] (A) to (D);
    \draw[double equal sign distance] (C) to (F);
   \end{tikzpicture}
\end{center}

Since $\pi_2f^* = \pi_2'$ we can take adjoints of both sides, yielding $f\pi_{2*} = \pi_{2*}'$, which in particular means that $f$ preserves the splitting. Thus, the functor described above is well defined. It clearly faithful and bijective on objects.

To see that the functor is full, recall that every morphism of adjoint extensions has a left adjoint, which is itself a morphism of extensions.
\end{proof}

\bibliographystyle{abbrv}  
\bibliography{references}

\end{document}